\documentclass[12pt]{article}
\usepackage[all]{xy}
\usepackage{graphicx}
\usepackage{amssymb}
\usepackage{amsfonts}

\newtheorem{theorem}{THEOREM}[section]
\newtheorem{Corollary}[theorem]{Corollary} %numerado
\newtheorem{lemma}[theorem]{Lemma}
\newtheorem{proposition}[theorem]{Proposition}

\newenvironment{remark}{\begin{trivlist}
\item[\hspace{\labelsep}{\noindent\em Remark:}]
}{\end{trivlist}}

\newenvironment{proof}{\begin{trivlist}
\item[\hspace{\labelsep}{\noindent\bf Proof.}]\rm
}{\end{trivlist}}

\def\cuadro{\hfill \mbox{$\square$}}

\newcommand{\bulito}{\ {\scriptstyle \bullet}\, \ }
\newcommand{\bulp}{{\scriptscriptstyle \bullet}}
\newcommand{\diagp}{\hbox{\scriptsize\rm diag}}
\renewcommand{\dim}{\hbox{\rm dim}}
\newcommand{\Hom}{\hbox{\rm Hom}}

\newcommand{\mod}{\hbox{\rm mod}}
\newcommand{\Mod}{\hbox{\rm Mod}}
\newcommand{\ord}{\hbox{\rm ord}}
\newcommand{\ordp}{\hbox{\scriptsize\rm ord}}
\newcommand{\rad}{\hbox{\rm rad}}
\newcommand{\unog}{\hbox{\bf 1}}
\newcommand{\id}{\hbox{\rm id}}

\def\aderecha#1{\smash{\mathop{\longrightarrow}\limits^{#1}}}

\newcommand{\N}{\mathbb{N}}

\begin{document}

\title{Covering functors without groups}

\date{ }
\author{Jos\'e A. de la Pe\~na and Mar\'\i a Julia Redondo
\thanks{The research for this paper was initiated during a visit of the second named author to UNAM, M\'exico.}}
\maketitle

\begin{abstract}
Coverings in the representation theory of algebras were introduced for the Auslander-Reiten quiver of a representation finite algebra in \cite{10} and later for finite dimensional algebras in \cite{2,4,5}. The best understood class of {\em covering functors} is that of    {\em Galois covering functors} $F: A \to B$ determined by the action of a group of automorphisms of $A$. In this work we introduce the {\em balanced covering functors} which include the Galois class and for which classical Galois covering-type results still hold. For instance, if $F:A \to B$ is a balanced covering functor, where $A$ and $B$ are linear categories over an algebraically closed field, and $B$ is tame, then $A$ is tame.   
\end{abstract}

\section*{Introduction and notation}

Let $k$ be a field and $A$ be a finite dimensional (associative with
$1$) $k$-algebra. One of the main goals of the {\em representation
theory of algebras\/} is the description of the category of finite
dimensional left modules ${_A\mod}$. For that purpose it is important
to determine the representation type of $A$. The finite
representation type (that is, when $A$ accepts only finitely many
indecomposable objects in ${_A\mod}$, up to isomorphism) is well
understood. In that context, an important tool is the construction of Galois
coverings $F\colon \tilde{A}\to A$ of $A$ since $\tilde{A}$ is a
locally representation-finite category if and only if $A$ is
representation-finite \cite{4,9}. For a tame algebra $A$ and a Galois covering 
$F\colon \tilde{A}\to A$, the category $\tilde{A}$ is also tame, but the converse does not hold \cite{7,15}.

Coverings were introduced in \cite{10} for the Auslander-Reiten
quiver of a representation-finite algebra. For algebras of
the form $A=kQ/I$, where $Q$ is a quiver and $I$ an admissible ideal
of the path algebra $kQ$, the notion of covering was introduced in
\cite{2,4,5}. Following \cite{2}, a functor $F\colon A\to
B$, between two locally bounded $k$-categories $A$
and $B$, is a {\em covering functor\/} if the following
conditions are satisfied:
\begin{itemize}
\item[(a)] $F$ is a $k$-linear functor which is onto on objects;
\item[(b)] the induced morphisms
$$\bigoplus\limits_{Fb'=j}A(a,b')\to B(Fa,j)\hbox{
and }\bigoplus\limits_{Fa'=i} A(a',b)\to B(i,Fb)$$
are bijective for all $i,j$ in $B$ and $a,b$ in $A$. 
\end{itemize}
We denote by
$({_{b'}f^{\bulp}_a})_{b'}\mapsto f$ and $({^\bulp_bf_{a'}})_{a'}\mapsto f$
the corresponding bijections. We shall consider $F_\lambda \colon
{_A\mod} \to {_B\mod}$ the left adjoint to the pull-up functor 
$F_{\bulp}\colon {_B\mod} \to {_A\mod}$, $M\mapsto M F$, where ${_C\mod}$ denotes the category of left modules over the $k$-category $C$, consisting of covariant $k$-linear functors.

The best understood examples of covering functors are the
{\em Galois covering functors\/} $A \to B$ given by the action of
a group of automorphisms $G$ of $A$ acting freely on objects and
where $F\colon A \to B=A/G$ is the quotient defined by the
action. See \cite{2,3,4,5,9} for results on Galois coverings. Examples of coverings which are not of Galois type will be exhibited in Section \ref{sec1}.

In this work we introduce {\em balanced coverings\/} as those coverings
$F\colon A \to B$ where ${_bf^{\bulp}_a} ={^{\bulp}_bf_a}$ for every 
$f\in B(Fa,Fb)$. Among many other examples, Galois coverings are balanced, see Section \ref{sec2}. We shall prove the following:

\begin{theorem}\label{teorema 1} %1
Let $F\colon A \to B$ be a balanced covering. Then every finitely generated
$A$-module $X$ is a direct summand of $F_{\bulp}F_\lambda X$.
\end{theorem}

In fact, according to the notation in \cite{1}, we show that a balanced covering
functor is a {\em cleaving\/} functor, see Section \ref{sec3}. This is essential for
extending Galois covering-type results to more general situations.
For instance we show the following result.

\begin{theorem}\label{teorema 2} %2
Assume that $k$ is an algebraically closed field and let $F\colon A \to B$ be a covering functor.
Then the following hold: 
\begin{itemize}
\item[{\rm (a)}] If $F$ is induced from a map $f:(Q,I) \to (Q',I')$ of quivers with relations, 
where $A=kQ/I$ and $B=kQ'/I'$, then $B$ is locally representation-finite if and only if so is $A$;
\item[{\rm (b)}] If $F$ is balanced and $B$ is tame, then $A$ is tame.
\end{itemize}
\end{theorem}

More precise statements are shown in Section \ref{sec4}. For a discussion on the representation type of algebras we refer to \cite{1,6,7,13,15}. 

\section{Coverings: examples and basic properties}\label{sec1} %1

\subsection{The pull-up and push-down functors} \label{subsec1.1} Following \cite{2,4}, consider a locally bounded $k$-category $A$, that is, $A$ has a (possibly infinite) set of non-isomorphic objects $A_0$ such that
\begin{itemize}
\item[(a)] $A(a,b)$ is a $k$-vector space and the composition corresponds to linear
maps $A(a,b)\otimes_k A(b,c)\to A(a,c)$ for every $a,b,c$ objects in
$A_0$;
\item[(b)] $A(a,a)$ is a local ring for every $a$ in $A_0$;
\item[(c)] $\sum\limits_bA(a,b)$ and $\sum\limits_bA(b,a)$ are finite
dimensional for every $a$ in $A_0$.
\end{itemize}

For a locally bounded $k$-category $A$, we denote by ${_A\Mod}$
(resp. $\Mod_A$) the category of covariant (resp. contravariant) functors $A \to \Mod_k$; by 
${_A\mod}$ (resp. $\mod_A$) we denote the full subcategory of locally finite-dimensional functors $A \to \mod_k$ of the category ${_A\Mod}$ (resp. $\Mod_A$). In case $A_0$ is finite, $A$ can be identified with the finite-dimensional $k$-algebra $\oplus_{a,b \in A_0} A(a,b)$; in this case the category ${_A\Mod}$ (resp. ${_A\mod}$) is equivalent to the category of left $A$-modules (resp. finitely generated left $A$-modules). 

According to \cite{13}, in case $k$ is algebraically closed, there exist a quiver $Q$ and an ideal $I$ of the path category $kQ$, such that $A$ is equivalent to the quotient $kQ/I$. Then any module $M \in {_A\Mod}$ can be identified with a {\em representation} of the {\em quiver with relations} $(Q,I)$. Usually our examples will be presented by means of quivers with relations. 

Let $F\colon A\to B$ be a $k$-linear functor between two locally
bounded $k$-categories. The {\em pull-up\/} functor $F_{\bulp}
\colon {_B\Mod} \to {_A\Mod}$, $M\mapsto M F$ admits a left
adjoint $F_\lambda \colon {_A\Mod} \to\break {_B\Mod}$, called the
{\em push-down\/} functor, which is uniquely defined (up to
isomorphism) by the following requirements:
\begin{itemize}
\item[(i)] $F_\lambda A(a,-)=B(Fa,-)$;
\item[(ii)] $F_\lambda$ commutes with direct limits. 
\end{itemize}
In particular, $F_\lambda$ preserves projective modules. 
Denote by $F_\rho \colon {_A\Mod} \to {_B\Mod}$ the right adjoint to $F_{\bulp}$. 

For covering functors $F\colon A\to B$ we get an explicit description
of $F_\lambda$ and $F_\rho$ as follows:

\begin{lemma} {\rm \cite{2}}.
Let $F\colon A\to B$ be a covering functor. Then
\begin{itemize}
\item[{\rm (a)}] For any $X\in  {_A\mod}$ and $f\in B(i,j)$,
$$F_\lambda X(f)=(X({_bf^\bulp_a}))\colon \bigoplus\limits_{Fa=i}X(a)
\to \bigoplus\limits_{Fb=j}X(b),\hbox{ with
}\sum_{Fb=j}F({_bf^\bulp_a})=f.$$
In particular, $F_\bulp (a,-)\colon F_\lambda A(a,-)\to B(Fa,-)$
is the natural isomorphism given by $({_bf^\bulp_a})_b\mapsto f$.
\item[{\rm (b)}] For any $X\in {_A\mod}$ and $f\in B(i,j)$
$$F_\rho X(f)=(X({^\bulp_bf_a}))\colon \prod\limits_{Fa=i}X(a)\to
\prod\limits_{Fb=j} X(b),\hbox{ with }\sum_{Fa=i}F({^\bulp_bf_a})=f.$$
In particular, $F_\bulp D(-,b)\colon F_\rho DA(-,b)\to DB(-,Fb)$ is
the natural isomorphism induced by $({^\bulp_bf_a})_a\mapsto
f$. \cuadro
\end{itemize}
\end{lemma}

\subsection{The order of a covering} The following lemma allows us to introduce the notion of {\em order} of a covering.

\begin{lemma}\label{lema1.3}
Let $F\colon A\to B$ be a covering functor. Assume that $B$ is connected and a fiber $F^{-1}(i)$ is finite, for some $i\in B_0$. Then the fibers have constant cardinality.
\end{lemma}

\begin{proof}
Let $i\!\in\!B_0$ and $0\!\ne\!f\!\in\!B(i,j)$. For $a\!\in\!F^{-1}(i)$,
$\sum\limits_{Fb=j}\dim_kA(a,b)=\dim_kB(i,j)$. Hence
$|F^{-1}(i)|\dim_kB(i,j)\!=\!$
$\sum\limits_{Fa=i}\sum\limits_{Fb=j}\!\dim_kA(a,b)\!=\!$
$\sum\limits_{Fb=j}\sum\limits_{Fa=i}\!\dim_kA(a,b)\!=\!$ \break
$|F^{-1}(j)|\dim_kB(i,j)$ and $|F^{-1}(i)|=|F^{-1}(j)|$. Since $B$ is connected, the claim follows. \cuadro
\end{proof}

In case $F\colon A\to B$ is a covering functor with $B$ connected and $A_0$ is finite, we define the {\em order\/} of $F$ as $\ord\,(F)=|F^{-1}(i)|$ for any
$i\in B_0$. Thus $\ord\,(F)|B_0|=|A_0|$.

We recall from the Introduction that a covering functor $F\colon A\to B$ is {\em
balanced\/} if ${_bf^\bulp_a}={^\bulp_bf_a}$ for every couple of objects
$a,b$ in $A$.

\begin{lemma} \label{igual}
Let $F\colon A\to B$ be a balanced covering functor, then $F_\lambda
=F_\rho$ as functors ${_A\mod} \to {_B\mod}$. 
\cuadro
\end{lemma}

\subsection{\it Examples}\label{subsec1.5} 

(a) Let $A$ be a locally bounded $k$-category and let $G$ be a group
of $k$-linear automorphisms acting freely on $A$ (that is, for $a\in
A_0$ and $g\in G$ if $ga=a$, then $g=1$). The quotient category $A/G$
has as objects the $G$-orbits in the objects of $A$; a
morphism $f\colon i\to j$ in $A/G$ is a family $f\colon ({_bf_a})\in
\prod_{a,b} A(a,b)$, where $a$ (resp. $b$) ranges in $i$ (resp. $j$) and
$g\cdot {_bf_a}={_{gb}f_{ga}}$ for all $g\in G$. The canonical projection
$F\colon A\to A/G$ is called a {\em Galois covering\/} defined by the
action of $G$.

A particular situation is illustrated by the following algebras (given as quivers with relations):

$$A\colon \,\raise36pt\hbox{$
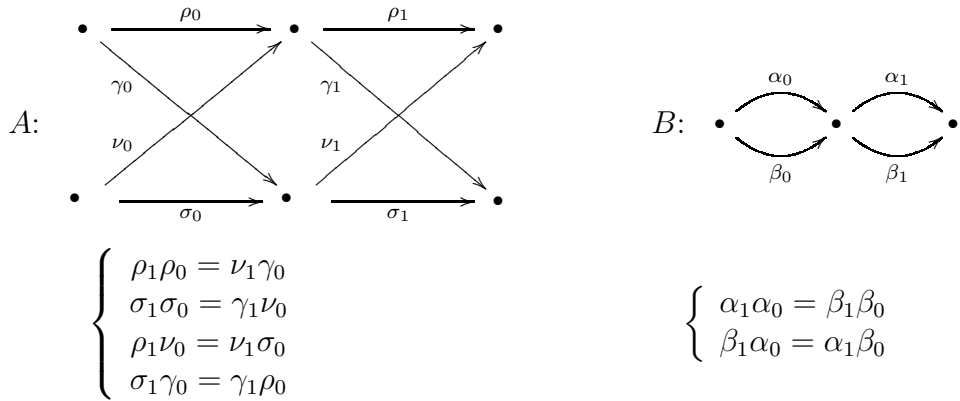
\xymatrix{
\bulito\ar[rr]^{\rho_0}\ar[ddrr]_(.25){\gamma_0}
 &&\bulito\ar[rr]^{\rho_1}\ar[ddrr]_(.25){\gamma_1}
 &&\bulito\\
&& &&\\
\bulito\ar[rr]_{\sigma_0}\ar[uurr]^(.25){\nu_0} \hole
 &&\bulito\ar[rr]_{\sigma_1}\ar[uurr]^(.25){\nu_1} \hole
 &&\bulito}$}\qquad\qquad B\colon \,
\xymatrix{
\bulito\ar@/^1pc/[r]^{\alpha_0}
\ar@/_1pc/[r]_{\beta_0} &\bulito\ar@/^1pc/[r]^{\alpha_1}
\ar@/_1pc/[r]_{\beta_1} &\bulito
}$$
$$\left\{\begin{array}{c}
\rho_1\rho_0 =\nu_1\gamma_0\\
\sigma_1\sigma_0=\gamma_1\nu_0\\
\rho_1\nu_0=\nu_1\sigma_0\\
\sigma_1\gamma_0=\gamma_1\rho_0
\end{array}\right.\qquad\qquad\qquad\qquad\qquad\qquad
\left\{\begin{array}{c}
\alpha_1\alpha_0=\beta_1\beta_0\\
\beta_1\alpha_0=\alpha_1\beta_0
\end{array}\right.$$

\bigskip

The algebra $A$ is tame, but $B$ is wild when char $k=2$ \cite{7}. The cyclic group $C_2$ acts freely on $A$ and $A/C_2$ is isomorphic to $B$.

(b) Consider the algebras given by quivers with relations and the
functor $F$ as follows:
$$
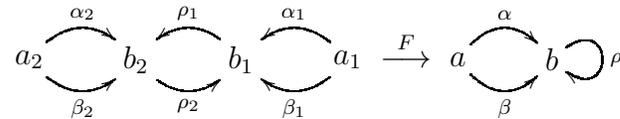
\xymatrix{
a_2\ar@/^1pc/[r]^{\alpha_2}\ar@/_1pc/[r]_{\beta_2} 
 &b_2\ar@/_1pc/[r]_{\rho_2} 
 &b_1\ar@/_1pc/[l]_{\rho_1}
 &a_1\ar@/_1pc/[l]_{\alpha_1}\ar@/^1pc/[l]^{\beta_1} 
}\ \aderecha{F}\ \xymatrix{
a\ar@/^1pc/[r]^{\alpha}
\ar@/_1pc/[r]_{\beta} &b\ar@(ur,dr)[]^\rho 
}$$
both algebras with rad $^2 =0$ and $F\alpha_1=\alpha$, $F\alpha_2=\alpha + \beta$, $F\beta_i=\beta$, $F\rho_i=\rho$, $i=1,2$. It is a simple exercise to check that $F$ is a balanced covering, but obviously it is not of Galois type.

(c) Consider the functor
$$A\colon \raise40pt\hbox{$
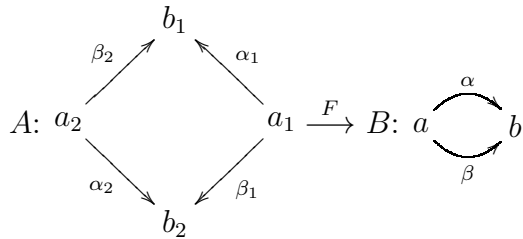
\xymatrix{
&b_1 &\\
a_2\ar[ur]^{\beta_2}\ar[dr]_{\alpha_2} 
 &&a_1\ar[ul]_{\alpha_1}\ar[dl]^{\beta_1}\\
&b_2 &
}$}\aderecha{F}\ B \colon \xymatrix{
a\ar@/^1pc/[r]^{\alpha}
\ar@/_1pc/[r]_{\beta} &b }$$
where $F\alpha_i=\alpha$, $i=1,2$, $F\beta_1=\beta$, $F\beta_2=\alpha
+\beta$. Since $F(\beta_2-\alpha_2)=\beta$ and
$F(\beta_1)=\beta$, then ${_{b_2}\beta^{\bulp}_{a_2}}=-\alpha_2$ and
${^{\bulp}_{b_2}\beta_{a_2}}=0$. Hence $F$ is a non-balanced covering
functor.

For the two dimensional indecomposable $A$-module $X$ given by $X(a_2)=k, X(b_2)=k, X(\alpha_2)=\id$ and zero otherwise, it follows that $F_{\bulp}F_\lambda X$ is indecomposable and hence $X$
is not a direct summand of $F_{\bulp}F_\lambda X$.

(d) As a further example, consider the infinite category $A$ and the balanced covering functor defined in the obvious way:
$$A\colon \xymatrix{
\cdots &\bulito\ar@/_1pc/[r]_{\beta'_2}
 &\bulito\ar@/_1pc/[l]_{\beta_2}\ar@/_1pc/[r]_{\alpha'_1}
 &\bulito\ar@/_1pc/[l]_{\alpha_1}\ar@/_1pc/[r]_{\beta'_1}
 &\bulito\ar@/_1pc/[l]_{\beta_1}\ar@(ur,dr)[]^{\alpha_0}
}\ \aderecha{F}\ \xymatrix{
\bulito\ar@(ul,dl)[]_\beta\ar@(ur,dr)[]^\alpha
}$$
where both categories $A$ and $B$ have $\rad^2=0$.

\subsection{Coverings of schurian categories}\label{subsec1.6} We say that a locally bounded $k$-category $B$ is {\em schurian\/} if
for every $i,j\in B_0$, $\dim_kB(i,j)\le 1$.

\begin{lemma}
Let $F\colon A\to B$ be a covering functor and assume that $B$ is
schurian, then $F$ is balanced.
\end{lemma}

\begin{proof}
Let $0\ne f\in B(i,j)$ and $Fa=i$, $Fb=j$. Since $B$ is schurian,
there is a unique $0\ne {_{b'}f^{\bulp}_a} \in A(a,b')$ with $Fb'=j$
and a unique $0\ne {^{\bulp}_bf_{a'}} \in A(a',b)$ with $Fa'=i$
satisfying $F{_{b'}f^{\bulp}_a}=f=F{^{\bulp}_bf_{a'}}$. In case
$b=b'$, then $a=a'$ and ${_bf^{\bulp}_a}={^{\bulp}_bf_a}$. Else $b\ne
b'$ and ${_bf^{\bulp}_a}=0$. In this situation $a \ne a'$ and
${^{\bulp}_bf_a}=0$. \cuadro
\end{proof}

\begin{proposition}\label{subsec1.7}  %1.7
Let $F\colon A\to B$ be a covering functor with finite order and
$B$ schurian. Then for every $M\in {_B\mod}$, $F_\lambda
F_{\bulp}M\cong M^{\ordp\,(F)}$.
\end{proposition}

\begin{proof}
For any $0 \ne f\in B(i,j)$ we get
$$\xymatrix{
F_\lambda F_{\bulp}M(i)=\bigoplus_{Fa=i}M(i)
 \ar[r]^-{\sim}\ar[d]_{(M(F{_bf^{\bulp}_a}))}
 &M^{\ordp\,(F)}(i)\ar[d]^{\diagp\,(M(f),\ldots,M(f))}\\
F_\lambda F_{\bulp}M(j)=\bigoplus_{Fb=j}M(j)
 \ar[r]^-{\sim} 
 &M^{\ordp\,(F)}(j)
}$$
Since for each $a$ there is a unique $b$ with ${_bf^{\bulp}_a} \ne 0$ such that $F{_bf^{\bulp}_a}=f$, then the square commutes. \cuadro
\end{proof}

\begin{remark}
If $B$ is not schurian the result may not hold as shown in
\cite[(3.1)]{7} for a Galois covering $F\colon B\to C$ with $B$ as in
Example (\ref{subsec1.5}.a).
\end{remark}

\subsection{Coverings induced from a map of quivers}\label{subsec1.8} Let $q\colon Q'\to Q$ be a {\em covering map of quivers}, that is,
$q$ is an onto morphism of oriented graphs inducing bijections $i^+ \to q(i)^+$ and $i^- \to q(i)^-$ for every vertex $i$ in $Q'$, where $x^+$ (resp. $x^-$) denotes those arrows $x \to y$ (resp. $y \to x$). For the concept of covering
and equitable partitions in graphs, see \cite{14}. 

Assume that $Q$ is a finite quiver. Let $I$ be an {\em admissible ideal\/} of the path algebra $kQ$, that
is, $J^n\subset I\subset J^2$ for $J$ the ideal of $kQ$ generated by
the arrows of $Q$. We say that $I$ is {\em admissible with respect\/}
to $q$ if there is an ideal $I'$ of the path category $kQ'$ such that
the induced map $kq \colon kQ \to kQ'$ restricts to isomorphisms
$ \bigoplus_{q(a)=i}I'(a,b) \to I(i,j)$ for $q(b)=j$ and 
$\bigoplus_{q(b)=j}I'(a,b) \to I(i,j)$ for $q(a)=i$.  
Observe that most examples in (\ref{subsec1.5}) (not Example (c)) are built according to
the following proposition:

\begin{proposition}\label{balanced}
Let $q\colon Q'\to Q$ be a covering map of quivers, $I$ an
admissible ideal of $kQ$ and $I'$ an admissible ideal of $kQ'$ making $I$ 
admissible with respect to $q$ as in the above definition. Then the induced functor
$F\colon kQ'/I'\to kQ/I$ is a balanced covering functor.
\end{proposition}

\begin{proof}
Since $q$ is a covering of quivers, it has the unique lifting
property of paths. Hence for any pair of vertices $i$ in $Q$ and $a$
in $Q'$ with $q(a)=i$, we have that
$$\xymatrix{
\bigoplus_{q(b)=j}kQ'(a,b)\ar[r]^-{k(q) \atop \sim}\ar[d]
 &kQ(i,j)\ar[d]\\
\bigoplus_{F(b)=j}kQ'/I'(a,b)\ar[r]^-{F} 
 &kQ/I(i,j)
}$$
is a commutative diagram with $F$ an isomorphism. This shows that
$F$ is a covering functor.

For any arrow $i\ \aderecha{\alpha}\ j$ in $Q$ and $q(a)=i$, there is
a unique $b$ in $Q'$ and an arrow $a\ \aderecha{\alpha'}\ b$ with
$q(\alpha')=\alpha$. Hence the class ${_bf^{\bulp}_a}$ of $\alpha'$ in
$kQ'/I'(a,b)$ satisfies that $F({_bf^{\bulp}_a})$ is the class
$f=\bar{\alpha}$ of $\alpha$ in $kQ/I(i,j)$. By symmetry,
${_bf^{\bulp}_a}={^{\bulp}_bf_a}$. For arbitrary $f\in kQ/I(i,j)$, $f$ is
the linear combination $\sum \lambda_if_i$, where $f_i$ is the
product of classes of arrows in $Q$. Observe that for arrows $i\
\aderecha{\alpha}\ j\ \aderecha{\beta}\ m$ we have
$_c(\bar{\beta}\bar{\alpha})^{\bulp}_a=
({_c\bar{\beta}^{\bulp}_b})({^{\bulp}_b\bar{\alpha}_a})=
({_c\bar{\beta}^{\bulp}_b})({_b\bar{\alpha}^{\bulp}_a})$. It follows that
$F$ is balanced. \cuadro
\end{proof}

In the above situation we shall say that the functor $F$ is {\em induced\/} from
a map $q\colon (Q',I')\to (Q,I)$ of quivers with relations.

\section{On Galois coverings}\label{sec2}%2

\subsection{Galois coverings are balanced} \label{subsec2.1}  %2.1
\begin{proposition}
Let $F\colon A\to B$ be a Galois covering, then $F$ is balanced.
\end{proposition}

\begin{proof}
Assume $F$ is determined by the action of a group $G$ of
automorphisms of $A$, acting freely on the objects $A_0$.
Let $i,j$ be objects of $B$ and $f\in B(i,j)$. Consider $a,b$ in $A$
with $Fa=i$, $Fb=j$ and $({_{b'}f^{\bulp}_a})_{b'}\in
\bigoplus\limits_{Fb'=j}A(a,b')$ with
$\sum\limits_{Fb'=j}F({_{b'}f^{\bulp}_a})=f$. \\
For each object $b'$ with $Fb'=j$, there is a unique $g_{b'}\in G$
with $g_{b'}(b')=b$. Then $(g_{b'}({_{b'}f^{\bulp}_a}))_{b'}\in
\bigoplus_{b'}A(g_{b'}(a),b)=\bigoplus_{Fa'=i}A(a',b)$
with $\sum_{b'}F(g_{b'}({_{b'}f^{\bulp}_a}))=
\sum_{b'}F({_{b'}f^{\bulp}_a})=f$. Hence
$g_{b'}({_{b'}f^{\bulp}_a})={^{\bulp}_bf_{g_{b'}(a)}}$ for every $Fb'=j$.
In particular, for $g_b=1$ we get ${_bf^{\bulp}_a}=\break {^{\bulp}_bf_a}$. \cuadro
\end{proof}

\subsection{The smash-product}\label{subsec2.2} We say that a $k$-category $B$ is {\em $G$-graded\/} with respect to
the group $G$ if for each pair of objects $i,j$ there is a vector
space decomposition $B(i,j)=\bigoplus_{g\in G}B^g(i,j)$ such
that the composition induces linear maps
$$B^g(i,j)\otimes B^h(j,m)\to B^{gh}(i,m).$$
Then the {\em smash product\/} $B\mathrel{\#}G$ is the $k$-category
with objects $B_0\times G$, and for pairs $(i,g),(j,h)\in B_0\times
G$, the set of morphisms is
$$(B\mathrel{\#}G)((i,g),(j,h))=B^{g^{-1}h}(i,j)$$ 
with compositions induced in natural way.

In \cite{8} it was shown that $B\mathrel{\#}G$ accepts a free action
of $G$ such that
$$(B\mathrel{\#}G)/G\ \aderecha{\sim}\ B.$$
Moreover, if $B=A/G$ is a quotient, then $B$ is a $G$-graded
$k$-category and 
$$(A/G)\mathrel{\#}G\ \aderecha{\sim}\ A.$$

\begin{proposition} \label{compatible}
Let $F\colon A\to B$ be a covering functor and assume that $B$ is a
$G$-graded $k$-category. Then
\begin{itemize}
\item[{\rm (i)}] Assume A accepts a $G$-grading compatible with $F$, that is, $F(A^g(a, b))\subseteq
B^g(Fa,Fb)$, for every pair $a, b \in A_0$ and $g \in G$. Then there is a covering
functor $F\mathrel{\#}G\colon
A\mathrel{\#}G\to B\mathrel{\#}G$ completing a commutative square
$$\xymatrix{
A\mathrel{\#}G\ar[r]^{F\mathrel{\#}G}\ar[d] 
 &B\mathrel{\#}G\ar[d]\\
A\ar[r]^F &B
}$$
where the vertical functors are the natural quotients. Moreover $F$ is balanced if and only if $F\mathrel{\#}G$ is
balanced.
\item[{\rm (ii)}] In case $B$ is a schurian algebra, then $A$ accepts a $G$-grading compatible with
$F$.
\end{itemize}
\end{proposition}

\begin{proof}
(i): For each $a,b\in A_0$, consider the decomposition 
$A(a,b)=\bigoplus_{g\in G}A^g(a,b)$ and 
$B(Fa,Fb)=\bigoplus_{g\in G}B^g(Fa,Fb)$. Since these decompositions are compatible with $F$,
then $A^g(a,b)=F^{-1}(B^g(Fa,Fb))$, for every $g\in G$. \\
For $\alpha \in (A\mathrel{\#}G)((a,g),(b,h))=A^{g^{-1}h}(a,b)=
F^{-1}(B^{g^{-1}h}(Fa,Fb))$, we have
$$(F\mathrel{\#}G)(\alpha)=F\alpha \in B^{g^{-1}h}(Fa,Fb)=
(B\mathrel{\#}G)((Fa,g),(Fb,h)).$$
(ii): Assume $B$ is schurian and take $a, b \in A_0$ and $g \in G$. Either $B^g(Fa,Fb) =
B(Fa,Fb) \not = 0$, if $A(a,b) \ne 0$ or $B^g(Fa,Fb) = 0$, correspondingly we set $A^g(a, b) = A(a, b)$ or
$A^g(a, b) = 0$. Observe that the composition induces linear maps $A^g(a, b)\otimes A^h(b, c) \to
A^{gh}(a, c)$, hence $A$ accepts a $G$-grading compatible with $F$. \cuadro
\end{proof}

\begin{remark}
In the situation above, the fact that $A$ and $B\mathrel{\#}G$ are connected
categories does not guaranty that $A\mathrel{\#}G$ is connected. For
instance, if $B=A/G$, then $A\mathrel{\#}G=A\times G$.
\end{remark}

The following result is a generalization of Proposition \ref{compatible}(ii).

\begin{proposition}\label{compatible2}
Let $F\colon A\to B$ be a (balanced) covering functor induced from a map
of quivers with relations. Let $F' \colon B' \to B$ be a Galois covering functor induced from a map
of quivers with relations defined by the action of a group $G$. Assume moreover that
$B'$ is schurian. Then $A$ accepts a $G$-grading compatible with $F$ making the following
diagram commutative
$$\xymatrix{
A\mathrel{\#}G\ar[r]^{\ \ F\mathrel{\#}G}\ar[d] 
 &B'\ar[d]^{F'}\\
A\ar[r]^F &B
}$$
\end{proposition}

\begin{proof}
Let $A=k\Delta /J$, $B=kQ/I$ and $B' = kQ'/I'$ be the corresponding presentations
as quivers with relations, $F$ induced from the map $\delta \colon \Delta \to Q$, while
$F'$ induced from the map $q \colon Q' \to Q$. For each vertex $a$ in $\Delta$ fix a vertex $a'$ in $Q'$
such that $F'a' = Fa$.\\
Consider an arrow $a\stackrel{\alpha}{\rightarrow} b$ in $\Delta$ and $\overline{\alpha}$ the corresponding element of  $A$. We claim
that there exists an element $g_\alpha \in G$ such that $F(\overline{\alpha}) \in B^{g_\alpha}(Fa,Fb)$. Indeed, we get
$F(\overline{\alpha})=\overline{\beta}= F'(\overline{\beta'})$ for arrows $Fa \stackrel{\beta}{\rightarrow} Fb$ and $a' \stackrel{\beta'}{\rightarrow} g_\alpha b'$ for a unique $g_\alpha \in G$.
Therefore $F(\overline{\alpha}) \in B^{g_\alpha}(Fa,Fb)$. We shall define $A^{g_\alpha}(a, b)$ as containing the space $k\overline{\alpha}$.
For this purpose, consider $g \in G$ and any vertices $a, b$ in $\Delta$, then $A^g(a, b)$ is the space
generated by the classes $\overline{u}$ of the paths $u : a \to b$ such that $F(\overline{u}) \in B^g(Fa,Fb)$. Since
the classes of the arrows in $\Delta$ generate $A$, then $A(a, b) = \bigoplus_{g\in G} A^g(a, b)$. We shall prove that there are linear maps
$$A^g(a, b) \otimes A^h(b, c) \to A^{gh}(a, c).$$
Indeed, if $\overline{u} \in A^g(a, b)$ and $\overline{v} \in A^h(b, c)$ for paths $u : a \to b$ and $v : b \to c$ in $\Delta$, let
$F(\overline{u}) = F'(\overline{u'})$ and $F(\overline{v}) = F'(\overline{v'})$ for paths $u' : a' \to gb'$ and $v' : b' \to hc'$ in $Q'$. Since
$B'$ is schurian then the class of the lifting of $F(\overline{vu})$ to $B'$ is $\overline{(gv')u'}$. Therefore
$$F(\overline{v})F(\overline{u}) = F'(\overline{(gv')u'}) \in B^{gh}(Fa,Fb).$$
By definition, the $G$-grading of $A$ is compatible with $F$. We get the commutativity
of the diagram from Proposition \ref{compatible}. \cuadro
\end{proof}

\subsection{Universal Galois covering}\label{subsec2.4} Let $B=kQ/I$ be a finite dimensional $k$-algebra. According to
\cite{5} there is a $k$-category $\tilde{B}=k\tilde{Q}/\tilde{I}$ and
a Galois covering functor $\tilde{F}\colon \tilde{B}\to B$ defined by
the action of the fundamental group $\pi_1(Q,I)$ which is {\em
universal\/} among all the Galois coverings of $B$, that is, for any
Galois covering $F\colon A\to B$ there is a covering functor
$F'\colon \tilde{B}\to A$ such that $\tilde{F}=FF'$. In fact, the following
more general result is implicitely shown in \cite{5}:

\begin{proposition} {\rm \cite{5}.} \label{universal}
The universal Galois covering $\tilde{F}\colon \tilde{B}\to B$ is
universal among all (balanced) covering functors $F\colon A\to B$
induced from a map $q\colon (Q',I')\to (Q,I)$ of quivers with
relations, where $A=kQ'/I'$.\cuadro
\end{proposition}

\section{Cleaving functors}\label{sec3} %3

\subsection{Balanced coverings are cleaving functors}\label{subsec3.1} Consider the $k$-linear functor $F\colon A\to B$ and the natural
transformation\break $F(a,b)\colon A(a,b)\to B(Fa,Fb)$ in two variables.
The following is the main observation of this work.

\begin{theorem}
Assume $F\colon A\to B$ is a balanced covering, then the natural
transformation $F(a,b)\colon A(a,b)\to B(Fa,Fb)$ admits a retraction
$E(a,b)\colon B(Fa,Fb)\to A(a,b)$ of functors in two variables $a,b$
such that $E(a,b)F(a,b)=\unog_{A(a,b)}$ for all $a,b\in A_0$.
\end{theorem}

\begin{proof}
Set $E(a,b)\colon B(Fa,Fb)\to A(a,b)$, $f\mapsto {^{\bulp}_bf_a}$ which
is a well defined map. For any $\alpha \in A(a,a')$, $\beta\in A(b,b')$,
we shall prove the commutativity of the diagrams:
$$\xymatrix{
B(Fa,Fb)\ar[r]^-{E(a,b)}\ar[d]_{B(Fa,F\beta)}
 &A(a,b)\ar[d]^{A(a,\beta)}\\ 
B(Fa,Fb')\ar[r]_-{E(a,b')} &A(a,b')
}\qquad\qquad \xymatrix{
B(Fa',Fb)\ar[r]^-{E(a',b)}\ar[d]_{B(F\alpha,Fb)}
 &A(a',b)\ar[d]^{A(\alpha,b)}\\ 
B(Fa,Fb)\ar[r]_-{E(a,b)} &A(a,b)
}$$
For the sake of clarity, let us denote by $\circ$ the composition of maps. 
Indeed, let $f\in B(Fa,Fb)$ and calculate 
$\sum_{Fa'=Fa}F(\beta \circ {^{\bulp}_bf_{a'}})=F\beta \circ f$,
hence
$$A(a,\beta)\circ E(a,b)(f)=\beta \circ {^{\bulp}_bf_a}=
{^{\bulp}_{b'}}(F\beta \circ f)_a =E(a,b')\circ B(Fa,F\beta )(f),$$
and the first square commutes. Moreover, let $h\in B(Fa',Fb)$ and
calculate\break $\sum_{Fb'=Fb}F({_{b'}h^{\bulp}_{a'}} \circ
\alpha)=h\circ F\alpha$ and therefore ${_bh^{\bulp}_{a'}} \circ \alpha
={_b(h\circ F\alpha)^{\bulp}_a}$. Using that $F$ is balanced we get
that $E(a,b)\circ B(Fa,Fb)(h)={^{\bulp}_b}(h\circ
F\alpha)_a={^{\bulp}_bh_{a'}} \circ \alpha =A(\alpha,b)\circ E(a',b)(h)$. \cuadro
\end{proof}

Given a $k$-linear functor $F\colon A\to B$ the composition
$F_{\bulp}F_\lambda \colon {_A\Mod} \to {_A\Mod}$ is connected to the
identity $\unog$ of ${_A\Mod}$ by a {\em canonical transformation\/}
$\varphi \colon F_{\bulp}F_\lambda \to \unog$\break determined by
$F_{\bulp}F_\lambda A(a,-)(b)=\bigoplus_{Fb'=Fb}A(a,b')\to
A(a,b)$, $(f_{b'})\mapsto f_b$, see \cite[page~234]{1}. Following
\cite{1}, $F$ is a {\em cleaving functor\/} if the canonical
transformation $\varphi$ admits a natural section $\varepsilon \colon \unog
\to F_{\bulp}F_\lambda$ such that $\varphi (X)\varepsilon
(X)=\unog_X$ for each $X\in {_A\Mod}$. The following statement, essentially from \cite{1}, yields Theorem \ref{teorema 1} in the Introduction.

\begin{Corollary} \label{cleaving}
Let $F\colon A\to B$ be a balanced covering, then $F$ is a cleaving
functor.
\end{Corollary}

\begin{proof}
Observe that $F_{\bulp}F_\lambda$ is exact, preserves direct sums and
projectives (the last property holds since $F_{\bulp}B(i,-)= \oplus_{Fa=i}A(a,-)$). Hence to define $\varepsilon \colon \unog \to F_{\bulp}F_\lambda$ it is enough 
to define $\varepsilon (A(a,-))\colon
A(a,-)\to F_{\bulp}F_\lambda A(a,-)$ with the desired
properties. For $b\in A_0$, consider $\varepsilon_b\colon A(a,b)\to
\bigoplus_{Fb'=Fb}A(a,b')=F_{\bulp}F_\lambda A(a,-)(b)$ the
canonical inclusion. For $h\in A(b,c)$ we shall prove the
commutativity of the following diagram:
$$\xymatrix{
A(a,b)\ar[r]^-{\varepsilon_b}\ar[d]_{A(a,h)}
 &\bigoplus_{Fb'=Fb}A(a,b')\ar[d]^{(A(a,{^{\bulp}_{c'}Fh_{b'}})}\\
A(a,c)\ar[r]^-{\varepsilon_{c}}
 &\bigoplus_{Fc'=Fc}A(a,c')
}$$
Let $f\in A(a,b)$, since $F$ is balanced
$A(a,{^{\bulp}_{c'}Fh_{b'}})\circ \varepsilon_b(f)=
{^{\bulp}_{c'}Fh_b} \circ f={_{c'}Fh^{\bulp}_b}\circ
f=\varepsilon_{c}\circ A(a,h)(f)$, since ${_{c'}Fh^{\bulp}_b} =h$ if $c'=c$ and 
it is $0$ otherwise. This is what we wanted to show. \cuadro
\end{proof}

\section{On the representation type of categories}\label{sec4} %4

\subsection{Representation-finite case}\label{subsec4.1} Recall that a $k$-category $A$ is said to be {\it locally representation-finite}
if for each object $a$ of $A$ there are only finitely many indecomposable $A$-modules $X$, up to isomorphism, such that $X(a) \ne 0$.
For a cleaving functor $F\colon A\to B$ is was observed in \cite{1}
that in case $B$ is of locally representation-finite then so is $A$. 
In particular this holds when $F$ is a Galois
covering by \cite{4}. We shall generalize this result for covering
functors. 

Part (a) of Theorem \ref{teorema 2} in the Introduction is the following:

\begin{theorem}
Assume that $k$ is algebraically closed and let $F\colon A\to B$ be a covering induced from a map of quivers with relations.
Then $B$ is locally representation-finite if and only if so is $A$. 
Moreover in this case the functor $F_\lambda
\colon {_A\mod} \to {_B\mod}$ preserves indecomposable modules and
Auslander-Reiten sequences.
\end{theorem}

\begin{proof}
Let $F\colon A\to B$ be induced from $q\colon (Q',I')\to (Q,I)$ where
$A=kQ'/I'$ and $B=kQ/I$. Let $\tilde{B}=k\tilde{Q}/\tilde{I}$ be the
universal cover of $B$ and $\tilde{F}\colon \tilde{B}\to B$ the
universal covering functor. By Proposition \ref{universal} there is a covering functor
$F'\colon \tilde{B}\to A$ such that $\tilde{F}=FF'$.

(1) Assume that $B$ is a connected locally representation-finite category. Since $F$ is induced by a map of quivers with relations, then Proposition \ref{balanced} implies that $F$ is balanced. Hence Corollary \ref{cleaving} implies that F is a cleaving functor. By \cite[(3.1)]{1}, $A$ is locally representation-finite; for the sake of completness, recall the simple argument: each indecomposable $A$-module $X\in {_A\mod}$ is a direct summand of $F_{\bulp}F_\lambda X=\bigoplus\limits^n_{i=1}F_{\bulp}N^{n_i}_i$ for a finite family $N_1,\ldots,N_n$ of representatives of the isoclasses $N$ of the indecomposable $B$-modules with $N(i) \ne 0$ for some $i=F(a)$ with $X(a) \ne 0$. 

(2) Assume that $A$ is a locally representation-finite category. First we show 
that $B$ is representation-finite. Indeed, by case (1),
since $F'\colon \tilde{B}\to A$ is a covering induced by a map of a
quiver with relations, then $\tilde{B}$ is locally
representation-finite. By \cite{9}, $B$ is representation-finite. In particular, \cite{4}
implies that $\tilde{F}_\lambda$ preserves indecomposable modules, hence $F_\lambda$ and $F'_\lambda$ also preserve indecomposable modules.

Let $X$ be an indecomposable $A$-module. We shall prove that $X$ is isomorphic to $F'_\lambda N$ for some indecomposable $\tilde B$-module $N$. Since indecomposable projective
$A$-modules are of the form $A(a,-)=F'_\lambda \tilde{B}(x,-)$ for
some $x$ in $\tilde{B}$, using the connectedness of $\Gamma_A$, we may assume that there is an irreducible
morphism $Y\ \aderecha{f}\ X$ such that $Y=F'_\lambda N$ for some
indecomposable $\tilde{B}$-module $N$. If $N$ is injective, say $N=D\tilde{B}(-,j)$, there is a surjective irreducible map $(h_i): N \to \oplus_i N_i$ such that all $N_i$ are indecomposable modules and $\xymatrix{0 \ar[r] & S_j \ar[r] & N \ar[r]^{(h_i)} & \oplus_i N_i \ar[r] & 0}$ is an exact sequence. Then 
$Y=DA(-,F'j)$ and the exact sequence 
$$\xymatrix{0 \ar[r] & S_{F'j} \ar[r] & Y \ar[r]^{(F'_\lambda(h_i))} & \oplus_i F'_\lambda(N_i) \ar[r] & 0}$$ 
yields the irreducible maps starting at $Y$ (ending at the indecomposable modules $F'_\lambda (N_i)$). Therefore $X=F'_\lambda(N_r)$ for some $r$, as desired. Next, assume that $N$ is not injective and consider the
Auslander-Reiten sequence $\xi \colon 0\ \aderecha{}\ N\
\aderecha{g}\ N'\ \aderecha{g'}\ N''\ \aderecha{}\ 0$ in
${_{\tilde{B}}\mod}$. We shall prove that the push-down $F'_\lambda \xi
\colon 0\ \aderecha{}\ F'_\lambda N\ \aderecha{F'_\lambda g}\
F'_\lambda N'\ \aderecha{F'_\lambda g'}\ F'_\lambda N''\ \aderecha{}\
0$ is an Auslander-Reiten sequence in ${_A\mod}$. This implies that there 
exists a direct summand $\bar{N}$ of $N'$ such that
$X\ \aderecha{\sim}\ F'_\lambda \bar{N}$ which completes the proof of the claim.

To verify that $F'_\lambda \xi$ is an Auslander-Reiten sequence, let $h\colon F'_\lambda N\to Z$ be non-split mono in ${_A\mod}$. 
Consider $\Hom_A(F'_\lambda N,Z)\ \aderecha{\sim}
\Hom_{\tilde{B}}(N,F'_{\bulp}Z)$, $h\mapsto h'$ which is not a split
mono (otherwise, then $\Hom_{\tilde{B}}(F'_{\bulp}Z,N)\
\aderecha{\sim} \Hom_A(Z,F'_\rho N)$, $\nu \mapsto \nu'$ with $\nu
h'=1_{F'_{\bulp}Z}$. By Lemma \ref{igual}, $F'_\lambda =F'_\rho$ and $\nu'h=1_Z$).
Then there is a lifting $\bar{h}\colon N'\to F'_{\bulp}Z$ with
$\bar{h}g=h'$. Hence $\Hom_{\tilde{B}}(N',F'_{\bulp}Z)\
\aderecha{\sim}\ \Hom_A(F'_\lambda N',Z)$, $\bar{h}\mapsto \bar{h}'$
with $\bar{h}'F'_\lambda g=h$.

We show that $F_\lambda$
preserves Auslander-Reiten sequences.  Let $X$ be an indecomposable $A$-module of the form $X=F'_\lambda
N$ for an indecomposable $\tilde{B}$-module $N$. Then $F_\lambda
X=F_\lambda F'_\lambda N=\tilde{F}_\lambda N$. Since by \cite{9},
$\tilde{F}_\lambda$ preserves indecomposable modules, then $F_\lambda
X$ is indecomposable. Finally, as above, we conclude that $F_\lambda$
preserves Auslander-Reiten sequences. \cuadro
\end{proof}

\subsection{Tame representation case}\label{subsec4.3} Let $k$ be an algebraically closed field. We recall that $A$ is said
to be of {\em tame representation type\/} if for each dimension $d\in
\N$ and each object $a\in A_0$, there are finitely many
$A-k[t]$-bimodules $M_1,\ldots,M_s$ which satisfy:

(a) $M_i$ is finitely generated free as right $k[t]$-module
$i=1,\ldots,s$;

(b) each indecomposable $X\in {_A\mod}$ with $X(a)\ne 0$ and
$\dim_kX=d$ is isomorphic to some module of the form
$M_i\otimes_{k[t]}(k[t]/(t-\lambda))$ for some $i\in \{1,\ldots,s\}$
and $\lambda \in k$.

In fact, it is shown in \cite{6} that $A$ is tame if (a) and (b) are
substituted by the weaker conditions:

(a') $M_i$ is finitely generated as right $k[t]$-module $i=1,\ldots,s$;

(b') each indecomposable $X\in {_A\mod}$ with $X(a)\ne 0$ and
$\dim_kX=d$ is a direct summand of a module of the form
$M_i\otimes_{k[t]}(k[t]/(t-\lambda))$ for some $i\in \{1,\ldots,s\}$
and $\lambda \in k$.

The following statement covers claim (b) of Theorem \ref{teorema 2} in the Introduction.

\begin{theorem}
Let $F\colon A\to B$ be a balanced covering functor. If $B$ is tame,
then $A$ is tame.
\end{theorem}

\begin{proof}
Let $a\in A_0$ and $d\in \N$. Let $M_1,\ldots,M_s$ be the
$B-k[t]$-bimodules satisfying (a) and (b): each indecomposable $M\in
{_B\mod}$ with $M(Fa)\ne 0$ and $\dim_kM\le d$ is isomorphic to some
$M_i\otimes_{k[t]}(k[t]/(t-\lambda))$ for some $i\in \{1,\ldots,s\}$
and $\lambda \in k$. \\
By Corollary \ref{cleaving} each indecomposable $X\in {_A\mod}$ with $X(a)\ne 0$ and
$\dim_kX=d$ is a direct summand of some
$F_{\bulp}(M_i\otimes_{k[t]}(k[t]/(t-\lambda)))$, which is isomorphic to $F_{\bulp}M_i\otimes_{k[t]}(k[t]/(t-\lambda))$,  for some $i\in
\{1,\ldots,s\}$ and $\lambda \in k$. Hence $A$ satisfies conditions (a') and (b'). \cuadro
\end{proof}

\noindent Jos\'e A. de la Pe\~na; Instituto de Matem\'aticas, UNAM. \\
Cd. Universitaria, M\'exico 04510 D.F.\\

{\it E-mail address:  jap@matem.unam.mx}

\medskip

\noindent Mar\'\i a Julia Redondo; Instituto de Matem\'atica,
Universidad Nacional del Sur. \\
Av. Alem 1253, (8000) Bah\'\i a Blanca, Argentina.\\

{\it E-mail address:  mredondo@criba.edu.ar}


\begin{thebibliography}{10000}
\bibitem{1} R. Bautista, P. Gabriel, A.$\!$~V. Roiter and L.
Salmer\'on.  
{\it Rep\-re\-sen\-ta\-tion-\-finite algebras and multiplicative
bases.}
Invent. Math. 81 (1984) 217--285. 

\bibitem{2} K. Bongartz and P. Gabriel.
{\it Covering spaces in Representation Theory.}
Invent. Math. 65 No. 3 (1982) 331--378.

\bibitem{11} K. Bongartz.
{\it Indecomposables are standard.}
Commentarii Math. Helvetici 60 (1985) 400-410.


\bibitem{8} C. Cibils and E. Marcos.
{\em Skew categories. Galois coverings and smash product of a
category over a ring.}
Proc. Amer. Math. Soc. 134 (2006) 39-50.

\bibitem{3} P. Dowbor and A. Skowro\'nski.
{\it Galois coverings of representation infinite algebras.}
Comment. Math. Helvetici 62 (1987) 311--337.

\bibitem{13} P. Gabriel.
{\it Auslander-Reiten sequences and representation-finite algebras.} 
Proc. ICRA II (Ottawa 1979). In Representation of Algebras, Springer
LNM 831 (1980) 1--71. 

\bibitem{4} P. Gabriel.
{\it The universal cover of a representation-finite algebra.}
Proc. Puebla (1980), Representation Theory I, Springer Lect. Notes
Math. 903 (1981) 68--105. 

\bibitem{12} P. Gabriel and J.$\!$ A. de la Pe\~na.
{\it Quotients of representation-finite algebras.}
Communications in Algebra 15 (1987) 279-307.

\bibitem{7} Ch. Geiss and J.$\!$~A. de la Pe\~na.
{\it An interesting family of algebras.}
Arch. Math. 60 (1993) 25--35.

\bibitem{14} C. D. Godsil.
\underline{Algebraic Combinatorics.}
Chapman and Hall Mathematics (1993).

\bibitem{5} R. Mart\'inez-Villa and J.$\!$~A. de la Pe\~na. 
{\it The universal cover of a quiver with relations.}
J. Pure and Applied Algebra 30 No. 3 (1983) 277--292.

\bibitem{9} R. Mart\'inez-Villa and J.$\!$~A. de la Pe\~na. 
{\it Auto\-mor\-phisms of rep\-re\-sen\-ta\-tion-\-finite algebras.} 
Invent. Math. 72 (1983) 359--362.

\bibitem{6} J.$\!$ A. de la Pe\~na.
{\it Functors preserving tameness.}
Fundamenta Math. 137 (1991) 177--185.

\bibitem{15} J.$\!$ A. de la Pe\~na. 
{\it On the dimension of the module-varieties of tame and wild algebras.} Communications in Algebra 19 (6), (1991) 1795-1807.

\bibitem{10} Ch. Riedtmann.
{\it Algebren, Darstellungsk\"ocher, \"Uberlagerungen und zur\"uch.} 
Comment. Math. Helv. 55 (1980) 199--224.


\end{thebibliography}
\end{document}